\documentclass{amsart}

\usepackage[T1]{fontenc}
\usepackage[vcentermath]{youngtab}

\newtheorem*{theo}{Theorem}

\newtheorem{theorem}{Theorem}
\newtheorem{lemma}{Lemma}
\newtheorem{proposition}{Proposition}

\theoremstyle{definition}
\newtheorem*{example}{Example}
\newtheorem*{remark}{Remark}
\newtheorem*{nota}{Notation}

\numberwithin{theorem}{section}
\numberwithin{proposition}{section} \numberwithin{lemma}{section}
\numberwithin{corollary}{section} 

\usepackage[dvips]{geometry}
\title[Singular components of Springer fibers]{Singular components of Springer fibers \\
in the two-column case}

\author{Lucas Fresse}

\address{Department of Mathematics, the Weizmann Institute of Science, Rehovot 76100, Israel}
\email{lucas.fresse@weizmann.ac.il}

  

\begin{document}
\begin{abstract}
We consider the Springer fiber ${\mathcal B}_u$
corresponding to a nilpotent endomorphism $u$ of nilpotent order $2$.
As a first result, we give a description of the elements
of a given component of ${\mathcal B}_u$ which are fixed by the action
of the standard torus relative to some Jordan basis of $u$.
By using this result,
we establish a necessary and sufficient condition of singularity
for the components of ${\mathcal B}_u$.
\end{abstract}

\maketitle

Let $V$ be a $\mathbb{C}$-vector space of dimension $n\geq 0$
and let $u:V\rightarrow V$ be a nilpotent endomorphism.
We denote by ${\mathcal B}_u$
the set of $u$-stable complete flags, i.e.
flags $(V_0\subset...\subset V_n=V)$
such that $u(V_i)\subset V_i$
for any $i$.
The set ${\mathcal B}_u$ is a projective subvariety of the variety
of complete flags.
The variety ${\mathcal B}_u$ is called {\em Springer fiber}
since it can be seen as the fiber over $u$ of the Springer resolution
of singularities of the cone of nilpotent endomorphisms of $V$
(see for example \cite{Slodowy}).
The Springer fiber ${\mathcal B}_u$ is not irreducible in general
and the geometry of its irreducible components has been an important topic of study
for more than thirty years.
Many problems remain unsolved, among them the problem to determine the singular
components of ${\mathcal B}_u$.

\medskip

The geometry of ${\mathcal B}_u$ depends on the Jordan form of $u$, which
can be represented by a Young diagram:
let $m_1\geq...\geq m_r$ be the sizes of the Jordan blocks of $u$, and
denote by $Y=Y(u)$ the Young diagram of rows of lengths $m_1,...,m_r$.
The problem to determine singular components
has a complete answer in two cases:
when the diagram $Y$ is of hook type or has two rows, 
every component of ${\mathcal B}_u$ is nonsingular (see \cite{Fung} and \cite{Vargas}).
When $Y$ has two columns, singular components can arise (see \cite{Spaltenstein}
or \cite{Vargas}).
In this article we give
a necessary and sufficient condition of singularity for the components
of ${\mathcal B}_u$
in the two-column case.

\medskip

In section 1, the diagram $Y$ is general.
Following \cite{Spaltenstein}, we recall the classical parameterization
of the components of ${\mathcal B}_u$ by the standard tableaux of shape $Y$.
Let ${\mathcal K}^T\subset {\mathcal B}_u$ be the component corresponding to the standard tableau $T$.
We fix a Jordan basis of $u$ and we denote by $H\subset GL(V)$ the torus
of diagonal automorphisms in this basis.
We show that the elements of ${\mathcal B}_u$ which are fixed by $H$ are parameterized by the
so-called row-standard tableaux, i.e. row-increasing numberings of $Y$
by $1,...,n$. Set $F_{T'}\in{\mathcal B}_u$ to be the flag corresponding to
the row-standard tableau $T'$.

\medskip

From section 2, we suppose that the diagram $Y$ has two columns of lengths $r\geq s$.
Let $\overline{T}$ be the tableau numbered from top to bottom by $1,...,r$
in the first column and by $r+1,...,n$ in the second column.
We show that the flag
$F_{\overline{T}}$ belongs to every component of ${\mathcal B}_u$, and that a
component ${\mathcal K}^T\subset{\mathcal B}_u$ is nonsingular if and only if $F_{\overline{T}}$ is
a nonsingular point of ${\mathcal K}^T$.
In section 2.2, we give some relations satisfied by the elements of ${\mathcal K}^T$.
For $T$ and $T'$ standard and row-standard tableaux of shape $Y$,
in section 2.3, we establish a necessary and sufficient condition for
$F_{T'}$ to be in ${\mathcal K}^T$.
This question can be related to the description of orbital varieties
in term of $B$-orbits given by A. Melnikov (see \cite{Melnikov}),
and the result that we prove here can be connected to \cite[Theorem 3.15]{Melnikov-Pagnon}. 

\medskip

Let ${\mathcal X}(Y)$ be the set of row-standard tableaux
which are obtained from $\overline{T}$ by switching two entries $i<j$, with $i\leq r$.
Let $\#A$ denote the number of elements in a set $A$.
Our main result, proved in section 3, is the following

\begin{theo}
Suppose that $Y=Y(u)$ has two columns.
Let $r$ be the length of the first column of $Y$.
The component ${\mathcal K}^T\subset {\mathcal B}_u$ is singular
if and only if
\[\#\{T'\in{\mathcal X}(Y):F_{T'}\in {\mathcal K}^T\}> r(r-1)/2.\]
\end{theo}

\medskip

\begin{nota} Fix some conventional notation.
Let $\mathbb{C}$ denote the field of complex numbers. In fact all our constructions
and results hold for any algebraically closed field.
Let $\#A$ be the number of elements in the set $A$.
In what follows 
flags will be denoted by $(V_0\subset...\subset V_n)$
or $(V_0,...,V_n)$ or $F$,
Young diagrams will be denoted by $Y,Y',...$
standard tableaux by $T,S,...$
and row-standard tableaux
by $T',T'',...$
Other pieces of notation will be introduced in what follows.
\end{nota}

\section{The geometry of ${\mathcal B}_u$ and the combinatorics of Young}

\subsection{Young diagram $Y(u)$}
\ 
\medskip

\label{section:diagrammes_Young}

Recall that a {\em Young diagram} is a collection of boxes displayed along
left-adjusted rows of decreasing lengths.
For example
\[Y=\yng(3,3,1)
\]
is a Young diagram with $7$ boxes.
Let $m_1\geq ...\geq m_r$ be the sizes of the Jordan blocks of the
endomorphism $u$, and let $Y(u)$ be the Young diagram of rows of lengths $m_1,...,m_r$.
This diagram has $n=\mathrm{dim}\,V$ boxes.
The dimension of the variety ${\mathcal B}_u$ depends on the diagram $Y(u)$
as shown by the following theorem (see \cite[\S II.5.5]{Spaltenstein}).

\begin{theo}
\label{theoreme_dimension_Bu}
The variety ${\mathcal B}_u$ is equidimensional.
Moreover,
denoting by $n_1,...,n_s$ the lengths of the columns of $Y(u)$,
we have
\[\mathrm{dim}\,{\mathcal B}_u=\sum_{q=1}^s\frac{n_q(n_q-1)}{2}\,.\]
\end{theo}

\bigskip

\noindent
{\em Two-column case.}
If $u^2=0$, then the diagram $Y(u)$ has (at most) two columns.
When $Y(u)$ has two columns of lengths $r$ and $s$,
we get \[\mathrm{dim}\,{\mathcal B}_u=\frac{r(r-1)}{2}+\frac{s(s-1)}{2}\,.\]

\subsection{Components of ${\mathcal B}_u$ parameterized by standard tableaux}

\label{section_standard}
\ 
\medskip

From now on, set for simplicity $Y=Y(u)$.

\medskip

A {\em standard tableau of shape $Y$} is a numbering of the boxes of $Y$
from $1$ to $n$ such that numbers increase in
each row from left to right and in each column from top to bottom.
We denote by ${\mathcal T}(Y)$ the set of standard tableaux of shape $Y$. \\
For example
\[T=\young(135,267,4)
\]
For $T\in{\mathcal T}(Y)$, 
let $T_{|i}$ for $i\leq n$ be the subtableau obtained by deleting boxes with numbers
$i+1,...,n$. The shape of the subtableau $T_{|i}$
is a subdiagram $Y_i^T\subset Y$ with $i$ boxes. The standard tableau $T$ can be regarded as the
maximal chain of subdiagrams
$\emptyset=Y_0^T\subset Y_1^T\subset...\subset Y_n^T=Y$.

\smallskip

Let $F=(V_0,...,V_n)\in{\mathcal B}_u$ be a $u$-stable flag.
For $i\in\{0,...,n\}$ the restriction map $u_{|V_i}:V_i\rightarrow V_i$
is a nilpotent endomorphism. Let $Y_i(F)=Y(u_{|V_i})$
be the Young diagram representing the Jordan form of $u_{|V_i}$ in the sense of section
\ref{section:diagrammes_Young}.
The diagrams $Y_0(F),...,Y_n(F)$ form an increasing sequence of subdiagrams
of $Y$. We set
\[{\mathcal B}_u^T=\{F\in{\mathcal B}_u:Y_i(F)=Y_i^T\ \forall i\}.\]
The subsets ${\mathcal B}_u^T$, for $T$ running over ${\mathcal T}(Y)$, form
a partition of ${\mathcal B}_u$.
Setting ${\mathcal K}^T=\overline{{\mathcal B}_u^T}$,
we have (see \cite[\S II.5.4--5]{Spaltenstein}):

\begin{theo}
\label{fait_Steinberg}
Let $T\in{\mathcal T}(Y)$ be standard.
The subset ${\mathcal K}^T\subset {\mathcal B}_u$
is
an irreducible component of ${\mathcal B}_u$.
Every component of ${\mathcal B}_u$ is obtained in that way.
\end{theo}

\subsection{Elements $F_{T'}\in{\mathcal B}_u$ parameterized by row-standard tableaux}
\ 
\medskip

\label{section_FT'}

\label{section_base_Jordan}

In this subsection, we fix a Jordan basis of $V$.
Since the lengths of the rows of $Y$ coincide with the sizes of the Jordan blocks of $u$,
it is possible to index the basis on the boxes of $Y$
so that the following is true.
Writing $e_x$ the vector of the basis associated to the box $x\in Y$,
we have $u(e_x)=0$ if $x$ is in the first column of $Y$,
and $u(e_x)=e_{x'}$ where $x'$ is the box just on the left of $x$ otherwise.

\smallskip

We call {\em row-standard tableau of shape $Y$}
a numbering of the boxes of $Y$
from $1$ to $n$ such that numbers increase in
each row from left to right.
We denote by ${\mathcal T}'(Y)$ the set of row-standard tableaux of shape $Y$. \\
For example
\[
T'=\young(256,137,4)
\]

For $T'\in{\mathcal T}'(Y)$,
the boxes $1,...,i$ of the subtableau $T'$
form a subset $X_i\subset Y$.
Let $F_{T'}=(V_0,V_1,...,V_n)$ be the flag defined by $V_i=\langle e_x:x\in X_i\rangle$ for every $i$.
It is easy to see that $F_{T'}$ belongs to ${\mathcal B}_u$.

Moreover,
let $T\in{\mathcal T}(Y)$ be standard and
suppose that each $i\in\{1,...,n\}$ belongs to the same column of $T$ and $T'$.
Then we can see that the flag $F_{T'}$ belongs to the
set ${\mathcal B}_u^T$.
Hence $F_{T'}$ belongs to the component ${\mathcal K}^T$. In particular,
considering $T$ as a row-standard tableau, we get $F_T\in{\mathcal B}_u^T$,
thus $F_T\in {\mathcal K}^T$.

\medskip

We prove the following

\begin{lemma}
\label{lemme_singularite_pointsfixes}
A component ${\mathcal K}\subset {\mathcal B}_u$ is nonsingular if and only
if every flag of the form $F_{T'}$ with $T'\in{\mathcal T}'(Y)$
which belongs to ${\mathcal K}$ is nonsingular.
\end{lemma}

\begin{proof}[Proof.]
The implication $\Rightarrow$ is immediate.
To prove the other implication, suppose there is a singular $F\in {\mathcal K}$.
Let $H\subset GL(V)$ be the subgroup of diagonal automorphisms with
respect to the basis. Observe that
the flags $F_{T'}$ for $T'$ running over ${\mathcal T}'(Y)$ are
the elements of ${\mathcal B}_u$
fixed by $H$ for its natural action on flags.
However, this action does not leave ${\mathcal B}_u$ invariant.
Let us construct a subtorus $H'\subset H$
with the same fixed points
and leaving ${\mathcal B}_u$ invariant.
To this end, we choose pairwise distinct numbers $\epsilon_x$
associated to the boxes $x\in Y$, so
that $\epsilon_{x'}=\epsilon_x+1$ if $x'$ is the box on the left of $x$.
For $t\in\mathbb{C}^*$, let $h_t\in GL(V)$ be defined by $h_t(e_x)=t^{\epsilon_x}e_x$
for any $x\in Y$.
Then $H'=(h_t)_{t\in\mathbb{C}^*}$ is a subtorus of $H$.
Since we have $h_tu=t^{-1}uh_t$ for any $t$,
the natural action of $H'$ on flags leaves invariant
the Springer fiber ${\mathcal B}_u$ and all its components.
As the $\epsilon_x$'s are pairwise distinct,
the flags $F_{T'}$ for $T'$ row-standard
are the $H'$-fixed points of ${\mathcal B}_u$. \\
The curve $\{h_tF:t\in\mathbb{C}^*\}$ admits a limit at the infinity
and this limit is necessarily a fixed point of $H'$, hence it is some $F_{T'}$
with $T'\in{\mathcal T}'(Y)$.
Since $F$ is singular, the point $h_tF$ is singular
for any $t$, and finally $F_{T'}$ is also a
singular point of the component.
\end{proof}

\section{Fixed points of the components in the $2$-column case}

\label{section_pointsfixes}

From now on, we suppose that the diagram $Y=Y(u)$ has two columns
of lengths $r$ and $s$, with $r\geq s$. As in section \ref{section_base_Jordan}
we fix a Jordan basis indexed on the boxes of $Y$.
For convenience we write $e_1,...,e_r$ the vectors of the basis associated to the boxes
of the first column of $Y$ from top to bottom, and $e_{r+1},...,e_n$ the vectors
associated to the boxes of the second column. Thus we have $u(e_i)=0$ for $i\leq r$
and $u(e_i)=e_{i-r}$ for $i\geq r+1$.

\subsection{Property of the flag $F_{\overline{T}}$ associated to the tableau $\overline{T}$}
\ 
\medskip

\label{section_Tbarre}

Let $\overline{T}$ be the tableau of shape $Y$ 
numbered from top to bottom by $1,...,r$
in the first column and by  $r+1,...,n$
in the second column:
\begin{center}
\begin{picture}(50,57)(0,0)
\put(10,26){$\overline{T}=\young(\ \ ,\ \ ,sn,\ ,r)$}
\put(34,47.7){1}
\put(42,47.7){r\!\mbox{\scriptsize$+$}\!1}
\put(35,35){\vdots}
\put(46,35){\vdots}
\put(35,12){\vdots}
\end{picture}
\end{center}
Thus for every $i$, the $i$-th subspace of the flag $F_{\overline{T}}$
is generated by the vectors $e_1,...,e_i$ of the basis.
The flag $F_{\overline{T}}$ satisfies the following properties.

\begin{proposition}
\label{proposition_singularite_FTbarre}
(a) The flag $F_{\overline{T}}$ belongs to every irreducible component
of ${\mathcal B}_u$. \\
(b) A component ${\mathcal K}\subset {\mathcal B}_u$ is nonsingular
if and only if $F_{\overline{T}}$ is nonsingular in ${\mathcal K}$.
\end{proposition}

\begin{proof}[Proof.]
Let $Z(u)\subset GL(V)$ be the centralizer of 
$I+u$. The natural action of $Z(u)$ on
flags leaves ${\mathcal B}_u$ invariant. Moreover, the group $Z(u)$ is
connected, since it is an open subset of the vector subspace of
$\mathrm{End}(V)$ formed by the endomorphisms which commute with
$u$. Hence $Z(u)$ leaves each component of ${\mathcal B}_u$ invariant.
Applying Lemma \ref{lemme_singularite_pointsfixes}, and using the
observation made just before Lemma \ref{lemme_singularite_pointsfixes}
that a component always contains some $F_{T'}$, it is sufficient to prove
that the flag $F_{\overline{T}}$ belongs to the closure of
the $Z(u)$-orbit of $F_{T'}$ for any $T'\in{\mathcal T}'(Y)$.
We reason by induction on the first entry which has not the same
place in $T'$ and $\overline{T}$. If $T'=\overline{T}$, then there
is nothing to prove. Suppose $T'\not=\overline{T}$, and take
$i\in\{1,...,n\}$ minimal which has not the same place in
$\overline{T}$ and $T'$. Write $F_{T'}=(V_0,...,V_n)$. We
have thus $V_{i-1}=\langle e_1,...,e_{i-1}\rangle$ and
$V_i=\langle e_1,...,e_{i-1},e_j\rangle$ for some $j>i$. For $t\in\mathbb{C}$ let
$w_t\in GL(V)$ be the automorphism such that $w_t(e_j)=e_j+te_i$
and $w_t(e_k)=e_k$ for $1\leq k\leq n$ with
$k\notin\{j-r,j,j+r\}$, and
\\{}\quad
\begin{tabular}{l}
- $w_t(e_{j-r})=e_{j-r}+te_{i-r}$ \ if $i\geq r+1$,\\
- $w_t(e_{j-r})=e_{j-r}$ \ if $i\leq r$ and $j\geq r+1$,\\
- $w_t(e_{j+r})=e_{j+r}+te_{i+r}$ \ if $j+r\leq n$. \end{tabular}\\
We have $w_t\in Z(u)$. Moreover, the curve $\{w_tF_{T'}:t\in\mathbb{C}\}$ admits a limit at the infinity which is the flag
$F_{\widetilde{T'}}$ associated to some tableau
$\widetilde{T'}\in{\mathcal T}'(Y)$ such that the entries
$1,...,i$ have the same
place in $\overline{T}$ and $\widetilde{T'}$. By induction, we
have $F_{\overline{T}}\in\overline{Z(u).F_{\widetilde{T'}}}$. Since $F_{\widetilde{T'}}\in\overline{Z(u).F_{{T'}}}$, we get
$F_{\overline{T}}\in\overline{Z(u).F_{T'}}$. The
proof is complete. 
\end{proof}

\subsection{Some relations satisfied by $F\in {\mathcal K}^T$}
\ 
\medskip

Let $F=(V_0,...,V_n)\in{\mathcal B}_u$. Let $0\leq i<j\leq n$.
The subspaces $V_i$ and $V_j$ are both invariant by $u$, hence the
quotient map $u_{|V_j/V_i}:V_j/V_i\rightarrow V_j/V_i$ can be
considered, and it is still a nilpotent map. First, we show the
following formula.

\begin{lemma}
\label{lemme_calcul_rang}
We have:
$\mathrm{rank}\,u_{|V_j/V_i}=\mathrm{dim}\,(V_i+u(V_j))-i$.
\end{lemma}

\begin{proof}[Proof.]
By applying the rank formula for the map
$V_j\rightarrow V_j/V_i,\,x\mapsto u(x)$, we get:
\[\mathrm{rank}\,u_{|V_j/V_i}=j-\mathrm{dim}\,u^{-1}(V_i)\cap V_j.\]
The map $u^{-1}(V_i)\cap V_j\rightarrow V_i\cap u(V_j),\,x\mapsto
u(x)$ is surjective and its kernel is $V_j\cap\mathrm{ker}\,u$,
hence
\[\mathrm{dim}\,u^{-1}(V_i)\cap V_j=\mathrm{dim}\,V_j\cap\mathrm{ker}\,u+\mathrm{dim}\,V_i\cap u(V_j).\]
On one hand, we have $\mathrm{dim}\,V_i\cap
u(V_j)=i+\mathrm{dim}\,u(V_j)-\mathrm{dim}\,(V_i+u(V_j))$. On the
other hand, the rank formula gives
$\mathrm{dim}\,V_j\cap\mathrm{ker}\,u=j-\mathrm{dim}\,u(V_j)$. The
lemma follows. 
\end{proof}

Let $T$ be standard. We associate to $T$ a row-standard tableau
$T^*$. Let $a_1<...<a_r$ (resp. $b_1<...<b_s$) be the entries of
the first (resp. second) column of $T$. We renumerate the entries of
the first column from $a_1^*$ to $a_r^*$:
\\[1mm]{}\quad
- Set $a_1^*=b_1-1$.\\[1mm]{}\quad - If $a_1^*,...,a_{p-1}^*$ have been
constructed for $p\in\{1,...,s\}$, then let $a_p^*$ be the maximal
element among the
$a\in\{a_1,...,a_r\}\setminus\{a_1^*,...,a_{p-1}^*\}$ such that
$a<b_p$. \\[1mm]{}\quad - For $p>s$, set
$a_p^*=\mathrm{Min}\,\{a_1,...,a_r\}\setminus\{a_1^*,...,a_{p-1}^*\}$.
\\[1mm] Then let $T^*$ be the tableau 
numbered from
top to bottom by $a_1^*,...,a_r^*$ 
 in the first column, and $b_1,...,b_s$ in the second
column. \\
For example
\[T=\young(13,25,4,6)
\qquad T^*=\young(23,45,1,6)
\]
As observed in section \ref{section_FT'}, 
since the content of the columns of $T$ and $T^*$ coincides,
the flag $F_{T^*}$ 
associated to $T^*$ belongs to the component ${\mathcal K}^T$. \\
Let $0\leq i<j\leq n$. Set
\[s_{j/i}^T=\#\{p:1\leq p\leq s\mbox{ and }i<a_p^*<b_p\leq j\}.\]
We prove the following

\begin{lemma}
\label{lemme_rang_generique}
Let $0\leq i<j\leq n$. The set
\[\{F=(V_0,...,V_n)\in {\mathcal K}^T:\mathrm{rank}\,u_{|V_j/V_i}=s_{j/i}^T\}\]
is a nonempty open subset of ${\mathcal K}^T$.
\end{lemma}

\begin{proof}[Proof.]
Recall that $Z(u)\subset GL(V)$ denotes the subgroup of elements
which commute with $u$. Observe that $\mathrm{rank}\,u_{|V_j/V_i}=s_{j/i}^T$
for any $F$ in the $Z(u)$-orbit of the flag $F_{T^*}$.
Thus, to prove the lemma, it is sufficient to prove 
that the $Z(u)$-orbit of the flag $F_{T^*}$
is an open subset of ${\mathcal K}^T$.
Let $Z(T^*)\subset Z(u)$ be the subgroup of elements
which fix the flag $F_{T^*}$.
It is sufficient to prove that
$\mathrm{dim}\,Z(T^*)=\mathrm{dim}\,Z(u)-\mathrm{dim}\,{\mathcal K}^T$. \\
Recall we have fixed a Jordan basis $(e_1,...,e_n)$ with $e_i=u(e_{i+r})$ for $i=1,...,s$.
Any $g\in Z(u)$ is thus determined by the images $ge_{s+1},...,ge_r$, which lie in $\mathrm{ker}\,u$,
and by $ge_{r+1},...,ge_n$.
Then we see that
$\mathrm{dim}\,Z(u)=(r-s)r+s(r+s)=r^2+s^2$.
By Theorem \ref{theoreme_dimension_Bu} we have 
$\mathrm{dim}\,{\mathcal K}^T=\mathrm{dim}\,{\mathcal B}_u={r(r-1)}/{2}+{s(s-1)}/{2}$.
Thus, we have to prove
\[
\mathrm{dim}\,Z(T^*)=\frac{r(r+1)}{2}+\frac{s(s+1)}{2}.
\leqno{\ \qquad(*)}
\]
We reason by induction on $n$, with immediate initialization for $n=1$.
Observe that $Z(T^*)$ has the same dimension as the vector space of endomorphisms
of $V$ commuting with $u$ and leaving stable each subspace in the flag $F_{T^*}$.
Then, using the definition of the flag $F_{T'}$ associated to a tableau $T'$, it is
straightforward to establish the formula
\begin{eqnarray}
\mathrm{dim}\,Z(T^*) & = & \#\{(p,q):1\leq p,q\leq r,\ s<q,\ a^*_p\leq a^*_q\} \nonumber \\
&& +\,\#\{(p,q):1\leq p\leq r,\ 1\leq q\leq s,\ a^*_p< b_q\} \nonumber \\
&& +\,\#\{(p,q):1\leq p\leq q\leq s,\ a^*_p\leq a^*_q\}. \nonumber
 \end{eqnarray}
We distinguish two cases. \\
A) Suppose that $1=a^*_p$ for some $p\leq s$.
Let $T^\flat$ be the tableau obtained from $T^*$
by removing the row containing $(1,b_p)$,
this tableau is row-standard up to 
moving $i\mapsto i-1$ if $i<b_p$ and $i\mapsto i-2$ if $i>b_p$.
By the above formula, we get
$\mathrm{dim}\,Z(T^*)=\mathrm{dim}\,Z(T^\flat)+r+s$.
By induction hypothesis, $\mathrm{dim}\,Z(T^\flat)=r(r-1)/2+s(s-1)/2$.
Thus $(*)$ follows. \\
B) Suppose that $1=a^*_p$ for $p>s$.
Let $T^\flat$ be the tableau obtained from $T^*$
by removing $1$,
this tableau is row-standard up to moving $i\mapsto i-1$, $\forall i=2,...,n$.
Likewise, the above formula implies
$\mathrm{dim}\,Z(T^*)=\mathrm{dim}\,Z(T^\flat)+r$.
By induction hypothesis, we have $\mathrm{dim}\,Z(T^\flat)=r(r-1)/2+s(s+1)/2$.
Again $(*)$ follows.
\end{proof}

We get now the following

\begin{proposition}
\label{proposition_inegalites}
Let $F=(V_0,...,V_n)\in{\mathcal B}_u$. 
If $F\in {\mathcal K}^T$, then \\
(a) $\mathrm{rank}\,u_{|V_j/V_i}\leq s_{j/i}^T$ for any $0\leq i<j\leq n$. \\
(b) $\mathrm{dim}(V_i+u(V_j))\leq s_{j/i}^T+i$ for any $0\leq i<j\leq n$.
\end{proposition}

\begin{proof}[Proof.]
By Lemma \ref{lemme_calcul_rang}, (a) and (b)
are equivalent. By Lemma \ref{lemme_calcul_rang} and Lemma \ref{lemme_rang_generique},
the set $\{F\in {\mathcal K}^T:\mathrm{dim}(V_i+u(V_j))= s_{j/i}^T+i\}$
is a nonempty open subset of ${\mathcal K}^T$.
Claim (b) follows from the lower semicontinuity
of the map $F\mapsto \mathrm{dim}(V_i+u(V_j))$.
\end{proof}

\subsection{A necessary and sufficient condition for $F_{T'}\in {\mathcal K}^T$}
\ 
\medskip

In this section, we characterize the row-standard tableaux $T'$ such that
the flag $F_{T'}$ belongs to the component ${\mathcal K}^T$.

\medskip

Let $T'\in{\mathcal T}'(Y)$ be row-standard.
Let $a'_1,...,a'_r$ (resp. $b'_1,...,b'_s$) 
be the entries from top to bottom in its first (resp. second) column.
For $0\leq i<j\leq n$, set 
\[s_{j/i}(T')=\#\{p:1\leq p\leq s\mbox{ and }i<a'_p<b'_p\leq j\}.\]
Then, writing $F_{T'}=(V_0,...,V_n)$, we have $\mathrm{rank}\,u_{|V_j/V_i}=s_{j/i}(T')$.
By Proposition \ref{proposition_inegalites}, we have
\[F_{T'}\in {\mathcal K}^T\ \Rightarrow\ s_{j/i}(T')\leq s_{j/i}^T\mbox{ for any $0\leq i<j\leq n$}.\]
The following theorem shows that this implication is in fact an equivalence.

\begin{theorem}
\label{theorem_pointsfixes}
Let $T'\in {\mathcal T}'(Y)$ be row-standard. The following conditions are equivalent. \\
(a) $F_{T'}\in {\mathcal K}^T$. \\
(b) $s_{j/i}(T')\leq s_{j/i}^T$ for any $0\leq i<j\leq n$.
\end{theorem}

\begin{proof}[Proof of Theorem \ref{theorem_pointsfixes}.] \

\smallskip

It remains to prove that claim (b) implies (a).
First, we introduce some notation and convention.
For $i$ in the first column of $T'$,
we denote by $\omega_{T'}(i)$ its right neighbor entry,
and put $\omega_{T'}(i)=\infty$ if $i$ has no entry on its right in $T'$.
We set $k< \infty$ for any integer $k$, so that
the set $\{1,...,n,\infty\}$ is totally ordered.
Recall that $Z(u)\subset GL(V)$ denotes the subgroup of elements which
commute with $u$. The natural action of $Z(u)$ on flags leaves
${\mathcal B}_u$ and its components invariant.

We know that the flag $F_{T'}$ belongs to the component ${\mathcal K}^T$
when the columns of $T$ and $T'$ have the same entries.
Now, let $i\in\{1,...,n\}$ be
the minimal entry which lies in different columns in $T$ and $T'$, and reason by induction
on $i$.
Since $s_{i/0}(T')\leq s_{i/0}^T$, the number $i$ belongs 
to the first column of $T'$ and to the second column of $T$.
Thus $T'$ contains strictly more entries $>i$ in its second column
than $T$. Since $s_{n/i}(T')\leq s_{n/i}^T$,
there is $j>i$ in the second column of $T'$
whose entry on its left is some $j'\leq i$.
Take $j$ minimal. \\
We distinguish two cases:

\smallskip
A) Suppose that there is $i'\in\{i+1,...,j\}$ in the first column of $T'$ 
such that $\omega_{T'}(i)<\omega_{T'}(i')$.
Take $i'$ minimal, and let $\widetilde{T'}$ denote the tableau obtained 
by switching $i$ and $i'$ in $T'$. This tableau is row-standard:
indeed, by minimality of $j$ we always have $j\leq\omega_{T'}(i)$,
hence $i'\leq\omega_{T'}(i)$.

\smallskip
B) Suppose that 
$\omega_{T'}(i')\leq \omega_{T'}(i)$ for any
$i'\in\{i+1,...,j\}$ in the first column of $T'$.
Then let $\widetilde{T'}$ denote the tableau obtained 
by switching $i$ and $j$ in $T'$. Let us prove that this tableau is row-standard.
Assume the contrary: $\omega_{T'}(i)\leq j$. Set $l=\omega_{T'}(i)$.
By the hypothesis every $i'\in\{i,...,l\}$ in the first column of $T'$
has a right neighbor entry among $i,...,l$.
By definition of $j$, every $i'\in\{i,...,l\}$ in the second column of $T'$
has a left neighbor entry among $i,...,l$.
Thus $l-i+1$ is even and $s_{j/i}(T')=(l-i+1)/2$.
On the other hand, since $i$ is in the second column of $T$, we have
$s_{j/i}^T<(l-i+1)/2$. We get $s_{j/i}(T')>s_{j/i}^T$, which contradicts the hypothesis.
Hence $j<\omega_{T'}(i)$, and the tableau ${\widetilde{T'}}$ is row-standard.

\medskip
\noindent
In both cases, we have defined ${\widetilde{T'}}$ row-standard.

\medskip
\noindent
{\sc Claim.} {\em The flag $F_{T'}$ lies to the closure of the $Z(u)$-orbit
of the flag $F_{\widetilde{T'}}$.}

\begin{proof}[Proof of the claim.]
Recall that we have fixed a
Jordan basis of $u$.
We renumber the vectors of the basis from $e'_1$ to $e'_n$
so that $e'_1,...,e'_l$ generate the $l$-th subspace of $F_{T'}$.
For $t\in \mathbb{C}$ we define an
automorphism $\phi_t:V\rightarrow V$. As above we distinguish two cases:
Suppose we are in case A above. 
For $k\in\{1,...,n\}$ different from $i'$ and
$\omega_{T'}(i')$, we set $\phi_t\,e'_k=e'_k$. Next, set
$\phi_t\,e'_{i'}=e'_{i'}+te'_i$. If $\omega_{T'}(i')<\infty$, then set in addition
$\phi_t\,e'_{\omega_{T'}(i')}=e'_{\omega_{T'}(i')}+te'_{\omega_{T'}(i)}$.
Suppose we are in case B above. For $k\in\{1,...,n\}$ different
from $j$, set $\phi_t\,e'_k=e'_k$. In addition set
$\phi_t\,e'_j=e'_j+te'_i$. 
In both cases $\phi_t\in Z(u)$,
and the flag $F_{T'}$ is the limit at the infinity of the curve
$\{\phi_t\,F_{\widetilde{T'}}:t\in\mathbb{C}\}$. The claim follows.
\end{proof}

By the claim, it is sufficient to prove that the flag $F_{\widetilde{T'}}$ belongs to ${\mathcal K}^T$,
and by induction hypothesis, it is sufficient
to prove that we have $s_{b/a}(\widetilde{T'})\leq s_{b/a}^T$
for any $0\leq a<b\leq n$.
As above, we distinguish two cases.

\medskip \noindent
A)
Suppose we are in the case A above.
Then $\widetilde{T'}$ is the tableau obtained by switching $i$ and $i'$
in the tableau $T'$.
By the minimality of $j$, we have $j\leq \omega_{T'}(i)$, hence
$i<i'<\omega_{T'}(i)<\omega_{T'}(i')$.
We have to show that $s_{b/a}(\widetilde{T'})\leq s_{b/a}^T$.
If $s_{b/a}(\widetilde{T'})=s_{b/a}(T')$, then it is immediate.
We have $s_{b/a}(\widetilde{T'})=s_{b/a}(T')$ unless we are in the following subcase (A.1).

\smallskip
\noindent
(A.1) We suppose that $i\leq a<i'<\omega_{T'}(i)\leq b<\omega_{T'}(i')$. \\
Then $s_{b/a}(\widetilde{T'})=s_{b/a}(T')+1$,
hence it is sufficient to prove $s_{b/a}(T')<s_{b/a}^T$.
As $\omega_{T'}(i)\leq b$ and as $i$ is in the second column of $T$,
we have
$s_{b/i}(T')<s_{b/i-1}(T')\leq s_{b/i-1}^T= s_{b/i}^T$.
Moreover we have $s_{a/i}(T')\leq s_{a/i}^T$ by hypothesis.
As $a<j$ and by minimality of $j$, the entry on the left of any $l\in\{i+1,...,a\}$ in the second column
of $T'$ also belongs to $\{i+1,...,a\}$.
As $a<i'<\omega_{T'}(i)\leq b$, the entry on the right of any $l\in\{i+1,...,a\}$ in the first column
of $T'$ also belongs to $\{i+1,...,b\}$.
Thus $s_{b/i}(T')=s_{b/a}(T')+(a-i-s_{a/i}(T'))$,
so that 
$s_{b/a}(T')=s_{b/i}(T')+s_{a/i}(T')-(a-i)<s_{b/i}^T+s_{a/i}^T-(a-i)$.
We have $s_{b/i}^T\leq s_{b/a}^T+(a-i-s_{a/i}^T)$,
since $a-i-s_{a/i}^T$ is the number of rows of the subtableau of $T^*$
of entries $i+1,...,a$.
The desired inequality follows.

\medskip \noindent
B)
Suppose we are in the case B above.
Then $\widetilde{T'}$ is the tableau obtained by switching $i$ and $j$
in the tableau $T'$.
Let $j'$ denote the entry on the left of $j$ in $T'$.
As already observed, we have $j<\omega_{T'}(i)$.
Moreover $j'<i$.
We have to show that $s_{b/a}(\widetilde{T'})\leq s_{b/a}^T$.
If $s_{b/a}(\widetilde{T'})=s_{b/a}(T')$, then it is immediate.
We have $s_{b/a}(\widetilde{T'})=s_{b/a}(T')$ unless we are in one of the following subcases.

\smallskip
\noindent
(B.1) We suppose that $j'<i\leq a<j<\omega_{T'}(i)\leq b$. \\
The proof is exactly the same as in the subcase (A.1)
with $j$ instead of $i'$.

\smallskip
\noindent
(B.2) We suppose that $a<j'<i\leq b<j<\omega_{T'}(i)$. \\
Then $s_{b/a}(\widetilde{T'})=s_{b/a}(T')+1$,
hence it is sufficient to prove $s_{b/a}(T')<s_{b/a}^T$.
By minimality of $j$, for any $l\in\{a+1,...,i-1\}$ in the first column of $T'$, 
we have $\omega_{T'}(l)\notin\{i,...,b\}$.
It follows $s_{b/a}(T')=s_{b/i-1}(T')+s_{i-1/a}(T')\leq s_{b/i-1}^T+s_{i-1/a}^T$.
Let $i^*$ be the entry on the left of $i$ in the tableau $T^*$.
If $i^*\in\{a+1,...,i-1\}$, then we have 
$s_{b/a}^T\geq s_{b/i-1}^T+s_{i-1/a}^T+1$ and the desired inequality ensues.
Suppose now that $i^*\leq a$.
By definition of $T^*$, for any $l\in\{a+1,...,i-1\}$ in the first column
of $T^*$, we have $l<\omega_{T^*}(l)<i$.
Hence $s_{i-1/a}^T$ is equal to the number of elements $l\in\{a+1,...,i-1\}$ in the first column
of $T$.
The entry $j'$ belongs to $\{a+1,...,i-1\}$ and to the first column of $T'$.
As $\omega_{T'}(j')=j>i$, the number
$s_{i-1/a}(T')$ is strictly lower that the number of elements $l\in\{a+1,...,i-1\}$ in the first column
of $T'$.
Moreover, by minimality of $i$,
every $l\in\{a+1,...,i-1\}$ is in the same column in $T$ and $T'$.
Thus
$s_{i-1/a}(T')<s_{i-1/a}^T$.
Therefore we get 
$s_{b/a}(T')<s_{b/i-1}^T+s_{{i-1}/a}^T\leq s_{b/a}^T$.
This completes the proof.
\end{proof}

\section{Characterization of singular components}

\subsection{Statement of the result}
\ 
\medskip

The diagram $Y=Y(u)$ is always supposed to have two columns of length $r\geq s$.
Let $\overline{T}$ be the tableau introduced in section \ref{section_Tbarre}.
We define ${\mathcal X}(Y)$ as the set of row-standard tableaux which are obtained from
$\overline{T}$ by switching two entries $i,j\in\{1,...,n\}$ such that $i\leq r$
and $i<j$.
\quad
For example, if $r=4$ and $s=2$, then the elements of ${\mathcal X}(Y)$
are the tableaux:
\[
\mbox{\scriptsize$\young(25,16,3,4)\quad\!
\young(35,26,1,4)\quad\!
\young(45,26,3,1)\quad\!
\young(15,36,2,4)\quad\!
\young(15,46,3,2)\quad\!
\young(15,26,4,3)\quad\!
\young(12,56,3,4)\quad\!
\young(13,26,5,4)\quad\!
\young(14,26,3,5)\quad\!
\young(15,23,6,4)\quad\!
\young(15,24,3,6)$}
\]
Observe that
${\mathcal X}(Y)$ is exactly the set of tableaux which are obtained from $\overline{T}$
by switching two entries $i,j$ such that
$i\leq r$ and $i<j< i+r$.

\medskip

Our main result is the following

\begin{theorem}
\label{theorem_principal}
Suppose $Y=Y(u)$ has two columns.
Let $r$ be the length of the first column of $Y$.
Let $T\in{\mathcal T}(Y)$ be standard.
The component ${\mathcal K}^T\subset {\mathcal B}_u$ is singular
if and only if we have 
\[\#\{T'\in {\mathcal X}(Y):F_{T'}\in {\mathcal K}^T\}> r(r-1)/2.\]
\end{theorem}

\begin{remark}
It will follow from Lemma \ref{lemme_dim_D}:
if the component ${\mathcal K}^T\subset {\mathcal B}_u$ is nonsingular,
then we have actually the equality $\#\{T'\in {\mathcal X}(Y):F_{T'}\in {\mathcal K}^T\}=r(r-1)/2$.
\end{remark}

\begin{example}
Let $T$ be the tableau
\[T=\young(13,25,4,6)
\]
After computation
using Theorem \ref{theorem_pointsfixes},
we obtain $\#\{T'\in {\mathcal X}(Y):F_{T'}\in {\mathcal K}^T\}=10>r(r-1)/2=6$.
Therefore, the component associated to $T$ is singular.
This singular component had already been pointed out by Vargas \cite{Vargas}.
Here are other examples of standard tableaux whose corresponding components
are singular:
\[
\young(13,25,47,6)\quad
\young(12,34,56,7)\quad
\young(13,25,47,68)\quad
\young(12,34,56,78)\quad
\young(14,26,3,5,7)\quad
\young(13,25,4,6,7)\quad\mbox{etc.}
\]
\end{example}

\subsection{Proof of Theorem \ref{theorem_principal}}
\ 
\medskip

By Theorem \ref{theoreme_dimension_Bu} we have
\[\mathrm{dim}\,{\mathcal K}^T=\mathrm{dim}\,{\mathcal B}_u=\frac{r(r-1)}{2}+\frac{s(s-1)}{2}\,.\]
By Proposition \ref{proposition_singularite_FTbarre},
the flag $F_{\overline{T}}$ belongs to the component ${\mathcal K}^T$ and,
to study the singularity of ${\mathcal K}^T$,
it is sufficient
to study the singularity of the element $F_{\overline{T}}$ in ${\mathcal K}^T$.
To do this, we compute the tangent space of ${\mathcal K}^T$
at $F_{\overline{T}}$.
Let us denote it by ${\mathcal D}$.
It is sufficient to establish the following

\begin{lemma}
\label{lemme_dim_D}
We have
\[\mathrm{dim}\,{\mathcal D}=\#\{T'\in {\mathcal X}(Y):F_{T'}\in {\mathcal K}^T\}+\frac{s(s-1)}{2}\,.\]
\end{lemma}

Let us prove the lemma.
Recall that we have fixed a Jordan basis parameterized by the
boxes of $Y$ (see section \ref{section_base_Jordan}).
As in the beginning of section \ref{section_pointsfixes},
we denote by $e_1,...,e_r$ the vectors of the basis associated to the boxes
of the first column of $Y$ and $e_{r+1},...,e_n$ the vectors
associated to the boxes of the second column, so that we have $u(e_i)=0$ for $i\leq r$
and $u(e_i)=e_{i-r}$ for $i\geq r+1$.
Let ${\mathcal B}$ be the variety of complete flags $F=(V_0\subset ...\subset V_n)$
and let $\Omega\subset {\mathcal B}$ be the open subvariety formed by the flags $F$
such that $V_i\not\subset V_{i-1}+\langle e_{i+1},...,e_n\rangle$ for any $i$.
For any $F\in\Omega$
and any pair $(i,j)$ such that $1\leq i<j\leq n$,
there are unique scalars $\phi_{i,j}(F)$ such that the vectors
\[f_i(F)=e_i+\sum_{j=i+1}^n\phi_{i,j}(F).e_j\qquad \mbox{for }i\in\{1,...,n\}\]
form a basis adapted to $F$
(i.e. $V_i=\langle f_1(F),...,f_i(F)\rangle$ for any $i$).
The maps $F\mapsto \phi_{i,j}(F)$ are algebraic
and they identify $\Omega$ to a vector space, whose
dual space is $\Omega^*=\bigoplus_{(i<j)}\mathbb{C}.\phi_{i,j}$
where the sum is taken over the pairs $(i,j)$ such that $1\leq i<j\leq n$.
Let $(\varepsilon_{i,j})$ be the dual basis of $\Omega$. 

The tangent space ${\mathcal D}$ is a vector subspace of $\Omega$.
Let ${\mathcal D}^\perp=\{\phi\in \Omega^*:\phi(\varepsilon)=0\ \forall \varepsilon\in{\mathcal D}\}$. 
To prove Lemma \ref{lemme_dim_D}, we use the following

\begin{lemma}
\label{lemme_D_Dorth}
Let $i,j$ be such that $1\leq i<j\leq n$. \\
(a) Assume $i> r$. Then $\varepsilon_{i,j}+\varepsilon_{i-r,j-r}\in{\mathcal D}$. \\
(b) Assume $j\geq i+r$. Then $\phi_{i,j}\in {\mathcal D}^\perp$.

\smallskip
\noindent
Assume now that $i\leq r$ and $j< i+r$. Let $T'$ be the tableau obtained 
from $\overline{T}$
by switching $i$ and $j$. The tableau $T'$ is row-standard and belongs to ${\mathcal X}(Y)$. \\
(c) If $F_{T'}\in {\mathcal K}^T$, then we have $\varepsilon_{i,j}\in{\mathcal D}$. \\
(d) Suppose that $F_{T'}\notin {\mathcal K}^T$. Then we have
$\phi_{i,j}-\phi_{i+r,j+r}\in{\mathcal D}^\perp$ if $j\leq s$, and $\phi_{i,j}\in{\mathcal D}^\perp$ if $j> s$.
\end{lemma}

\begin{proof}[Proof of Lemma \ref{lemme_dim_D}.]
We have $\mathrm{dim}\,{\mathcal D}+\mathrm{dim}\,{\mathcal D}^\perp=\mathrm{dim}\,\Omega=n(n-1)/2$.
Claims (a) and (c) of Lemma \ref{lemme_D_Dorth} imply
$\mathrm{dim}\,{\mathcal D}\geq \#\{T'\in {\mathcal X}(Y):F_{T'}\in {\mathcal K}^T\}+s(s-1)/2$.
Claims (b) and (d) of Lemma \ref{lemme_D_Dorth} imply
$\mathrm{dim}\,{\mathcal D}^\perp\geq n(n-1)/2-\#\{T'\in {\mathcal X}(Y):F_{T'}\in {\mathcal K}^T\}-s(s-1)/2$.
Thus these two inequalities are in fact two equalities and
Lemma \ref{lemme_dim_D} follows. \end{proof}

It remains to prove Lemma \ref{lemme_D_Dorth}.
Recall from the proof of Proposition \ref{proposition_singularite_FTbarre} that the group $Z(u)$ acts on ${\mathcal B}_u$
and leaves the component ${\mathcal K}^T$ (and any other component) invariant.

\begin{proof}[Proof of Lemma \ref{lemme_D_Dorth}.(a).]
For $t\in\mathbb{C}$,
let $w_t\in GL(V)$ be the automorphism defined by
$w_t:e_i\mapsto e_i+te_j$ and $w_t:e_{i-r}\mapsto e_{i-r}+te_{j-r}$,
and $w_t:e_k\mapsto e_k$ if $k\notin\{i,i-r\}$.
We have $w_t\in Z(u)$,
hence $w_t$ leaves ${\mathcal K}^T$ invariant.
For any $t\in\mathbb{C}$ we have thus $w_t.F_{\overline{T}}\in {\mathcal K}^T$.
Observe that this flag belongs to $\Omega$ and can be written $t.(\varepsilon_{i,j}+\varepsilon_{i-r,j-r})$.
Thus $\varepsilon_{i,j}+\varepsilon_{i-r,j-r}\in {\mathcal D}$.
\end{proof}

\begin{proof}[Proof of Lemma \ref{lemme_D_Dorth}.(b).]
Let $F\in \Omega\cap {\mathcal K}^T$.
Set for simplicity $f_k=f_k(F)$ and $\phi_{k,l}=\phi_{k,l}(F)$.
We have $u(f_i)\in\langle f_1,...,f_{i-1}\rangle$.
Thus the family $(f_1,...,f_{i-1},$ $u(f_i))$ has rank $i-1$.
Consider the matrix of this family in the basis $(e_1,...,e_n)$.
We express the minor of this matrix relative to the subbasis $(e_1,...,e_{i-1},e_{j-r})$
\quad (the symbol ``$*$'' equals $\phi_{k,l}$ for some pair $k<l$):
\[
\begin{array}{|cccc|}
1 & \ & (0) & * \\
\ & \ddots & \ & \vdots \\
(*) & \ & 1 & * \\
* & \cdots & * & \phi_{i,j}
\end{array}=\phi_{i,j}-P
\]
where $P$ is a polynomial on the $\phi_{k,l}$
with only terms of degree $\geq 2$.
By hypothesis this minor is equal to zero, thus we get:
$\phi_{i,j}=P$, hence the equation $\phi_{i,j}=0$ is satisfied
in the tangent space ${\mathcal D}$.
\end{proof}

\begin{proof}[Proof of Lemma \ref{lemme_D_Dorth}.(c).]
For $t\in\mathbb{C}^*$, let $w_t\in GL(V)$ be the automorphism defined by
$w_t: e_j\mapsto e_j+t^{-1}e_i$
and $w_t:e_k\mapsto e_k$ for $k\notin\{j,j+r\}$
and in addition $w_t:e_{j+r}\mapsto e_{j+r}+t^{-1}e_{i+r}$ in case $j\leq s$.
Since $w_t$ belongs to the group $Z(u)$,
it leaves ${\mathcal K}^T$ invariant.
Hence for every $t\in\mathbb{C}^*$ we have $w_t.F_{T'}\in {\mathcal K}^T$.
Observe that this flag belongs to $\Omega$ and can be written $t.\varepsilon_{i,j}$.
Thus $\varepsilon_{i,j}\in {\mathcal D}$.
\end{proof}

\begin{proof}[Proof of Lemma \ref{lemme_D_Dorth}.(d).]
We suppose that the flag $F_{T'}$ does not belong
to the component ${\mathcal K}^T$. By Theorem \ref{theorem_pointsfixes},
there are $0\leq a<b\leq n$ such that $s_{b/a}(T')>s_{b/a}^T$.
Observe that
$s_{b/a}(T')\leq s_{b/a}(\overline{T})+1$,
and that $s_{b/a}(\overline{T})\leq s_{b/a}^T$
since the flag $F_{\overline{T}}$ belongs to the component.
Thus
$s_{b/a}(T')= s_{b/a}(\overline{T})+1$
and
$s_{b/a}^T=s_{b/a}(\overline{T})=\mathrm{Max}(0,b-a-r)$.
By Proposition \ref{proposition_inegalites}, we get that
each flag $F=(V_0,...,V_n)\in {\mathcal K}^T$ satisfies the inequality
\[\mathrm{dim}(V_a+u(V_b))\leq \mathrm{Max}(a,b-r).\]
Let $F\in \Omega$
and set for simplicity $f_k=f_k(F)$ and $\phi_{k,l}=\phi_{k,l}(F)$.
We distinguish two cases.

\smallskip
\noindent
(A) First case: $j\leq r$. \\
Observe that, for having
$s_{b/a}(T')>s_{b/a}(\overline{T})$,
it is necessary to have:
\[i\leq a<j\leq i+r\leq b<j+r.\]
Set $i'=i+r$ and $j'=j+r$.
In the case $a\geq b-r$ we consider the matrix of the family
$(f_1,...,f_{i-1},f_{i+1},...,f_a,f_i,u(f_{i'}))$
in the basis $(e_1,...,e_n)$.
In the case $a<b-r$ we consider the matrix of the family
$(u(f_{r+1}),...,u(f_{i'-1}),$ $u(f_{i'+1}),...,u(f_b),f_i,u(f_{i'}))$
in the basis $(e_1,...,e_n)$.
Set $c=\mathrm{Max}(a,b-r)$.
In both cases we express the minor of the matrix relatively to the subbasis $(e_1,...,e_{i-1},e_{i+1},...,e_{c},e_i,e_{j})$
\quad (the symbol ``$*$'' equals $\phi_{k,l}$ for some pair $k<l$)
\quad (if $j'>n$, then set $\phi_{i',j'}=0$ by convention):
\[
\begin{array}{|ccccc|}
1 & \ & (0) & * & * \\
\ & \ddots & \ & \vdots & \vdots \\
(*) & \ & 1 & * & * \\
* & \cdots & * & 1 & 1 \\
* & \cdots & * & \phi_{i,j} & \phi_{i',j'}
\end{array}=\phi_{i,j}-\phi_{i',j'}-P
\]
where $P$ is a polynomial on the $\phi_{k,l}$
with only terms of degree $\geq 2$.
This minor is equal to zero, hence we get
$\phi_{i,j}-\phi_{i',j'}=P$. Therefore the relation $\phi_{i,j}=\phi_{i',j'}$ holds
in the tangent space ${\mathcal D}$.

\smallskip
\noindent
(B) Second case: $j\geq r+1$. \\
Observe that, for having
$s_{b/a}(T')>s_{b/a}(\overline{T})$,
it is necessary to have:
\[
a<j-r<i\leq b<j<i+r\quad\mbox{or}\quad j-r<i\leq a<j<i+r\leq b.\]
Set $i'=i+r$ and $j'=j-r$. Set $c=\mathrm{Max}(a,b-r)$. \\
First, suppose that $a<j-r<i\leq b<j<i+r$.
In the case $a\geq b-r$ consider the matrix of the family
$(f_1,...,f_{a},u(f_{i}))$
in the basis $(e_1,...,e_n)$.
In the case $a<b-r$ consider the matrix of the family
$(u(f_{r+1}),...,u(f_{b}),u(f_i))$
in the basis $(e_1,...,e_n)$. \\
Next, suppose that $j-r<i\leq a<j<i+r\leq b$.
In the case $a\geq b-r$ consider the matrix of the family
$(f_1,...,f_{a},f_{i})$
in the basis $(e_1,...,e_n)$.
In the case $a<b-r$ consider the matrix of the family
$(u(f_{r+1}),...,u(f_{b}),f_i)$
in the basis $(e_1,...,e_n)$.

\smallskip
\noindent
We express the minor of the considered matrix, relative to some subbasis.
For $a<j-r<i\leq b<j<i+r$, we choose the subbasis
$(e_1,...,e_{c},e_{j'})$.
For $j-r<i\leq a<j<i+r\leq b$, we choose the subbasis
$(e_1,...,e_{c},e_{j})$.
In both cases, the minor has the following expression
\quad (the symbol ``$*$'' equals $\phi_{k,l}$ for some pair $k<l$):
\[
\begin{array}{|cccc|}
1 & \ & (0) & * \\
\ & \ddots & \ & \vdots \\
(*) & \ & 1 & * \\
* & \cdots & * & \phi_{i,j}
\end{array}=\phi_{i,j}-P
\]
where $P$ is a polynomial on the $\phi_{k,l}$
with only terms of degree $\geq 2$.
This minor is equal to zero, hence we get
$\phi_{i,j}=P$. Therefore the equation $\phi_{i,j}=0$ holds
in the tangent space ${\mathcal D}$.
Our proof is now complete.
\end{proof}


\begin{thebibliography}{gz}
\bibitem{Fung}
F.Y.C. Fung, {\em On the topology of components of some Springer fibers and their relation to Kazhdan-Lusztig theory}.
Adv. Math. {\bfseries 178} (2003) 244--276.

\bibitem{Melnikov}
A. Melnikov, {\em Description of $B$-orbit closures of order 2 in
upper triangular matrices}. Transf. Groups {\bfseries 11} (2006) 217--247.

\bibitem{Melnikov-Pagnon}
A. Melnikov, N.G.J. Pagnon, {\em Reducibility of the intersections of components of a Springer fiber}.
Indag. Math., to appear.

\bibitem{Spaltenstein}
N. Spaltenstein, {\em Classes unipotentes et sous-groupes de Borel}. Lect. Notes in Math.,
vol. 946, Springer-Verlag, Berlin-New York, 1982.

\bibitem{Slodowy}
P. Slodowy, {\em Four lectures on simple groups and singularities}.
Communications of the Mathematical Institute, vol. 11, 
Rijksuniversiteit Utrecht, 1980.

\bibitem{Vargas}
J.A. Vargas, {\em Fixed points under the action of unipotent elements of $SL(n)$ in the flag variety}.
Bol. Soc. Mat. Mexicana {\bfseries 24} (1979) 1--14.
\end{thebibliography}
\end{document}